\documentclass[11pt]{article}
\usepackage{amsmath, amssymb, amsfonts, amstext, amsthm, textcomp, enumerate}
\usepackage[mathscr]{euscript}
\usepackage{float}
\usepackage{booktabs}
\usepackage{mathtools}
\usepackage[left=20mm,top=0.5in,bottom=15mm]{geometry}
\usepackage{graphicx}
\usepackage{caption}
\usepackage{epstopdf}
\usepackage{longtable}
\usepackage[utf8]{inputenc}
\usepackage{xcolor}
\usepackage{color}
\usepackage{hyperref}
\usepackage{graphicx}
\usepackage{dcolumn}
\usepackage{bm}
\usepackage{epstopdf}
\usepackage[english]{babel}
\usepackage{subfigure}
\usepackage{color}
\usepackage{ulem}

\newtheorem{thm}{Theorem}[section]
\newtheorem{lem}[thm]{Lemma}
\newtheorem{pro}[thm]{Proposition}

\newtheorem{rem}[thm]{Remark}
\newtheorem{ex}[thm]{Example}

\newcommand{\mnf}{\mathbf{M}_n\,(\mathbb{F})}
\newcommand{\mnr}{\mathbf{M}_n\,(\mathbb{R})}
\newcommand{\mnc}{\mathbf{M}_n\,(\mathbb{C})}

\newfont{\bb}{msbm10}

\begin{document}
\allowdisplaybreaks
	\title{Linear Preservers of Real Matrix Classes Admitting a Real Logarithm
    }
    \author{Shaun Fallat\thanks{Department of Mathematics and Statistics,
University of Regina, Regina, SK, Canada \\
(sfallat@uregina.ca).}
    \and  Samir Mondal\thanks{Department of Mathematics and Statistics,
University of Regina, Regina, SK, Canada 
(Samir.Mondal@uregina.ca).} }
    \date{}
 \maketitle

\begin{abstract}
In real Lie theory, matrices that admit a real logarithm reside in the identity component $\mathrm{GL}_n(\mathbb{R})_+$ of the general linear group $\mathrm{GL}_n(\mathbb{R})$, with logarithms in the Lie algebra $\mathfrak{gl}_n(\mathbb{R})$. The exponential map
\[
\exp : \mnr \to \mathrm{GL}_n(\mathbb{R})
\]
provides a fundamental link between the Lie algebra and the Lie group, with the logarithm as its local inverse.

In this paper, we characterize all bijective linear maps $\varphi : \mnr \to \mnr$ that preserve the class of matrices admitting a real logarithm (principal logarithm). 
We show that such maps are exactly those of the form
\[
\varphi(A) = c\, P A P^{-1}
\quad \text{or} \quad
\varphi(A) = c\, P A^{T} P^{-1},
\]
for some $P \in \mathrm{GL}_n(\mathbb{R})$ and $c > 0$.

The proof proceeds in two stages. First, we analyze preservers within the class of standard linear transformations. Second, using Zariski denseness, we prove that any bijective linear map preserving matrices with real logarithms (principal logarithm) must preserve $\mathrm{GL}_n(\mathbb{R})$, which then implies the map is of the standard form.

\end{abstract}

\vspace{1cm}

\noindent AMS-MSC: 15A86, 15A16 (primary); 54C05, 15B48 (secondary)

\section{Introduction and Preliminaries}\label{introduction}



Linear preserver problems continue to be an active research area in matrix and operator theory for more than a century. They focus on describing a general form of all linear transformations defined on a matrix algebra that preserve certain functions, subsets, relations, or structural properties, etc.. 


Frobenius's trans-formative 1897 paper \cite{Frobenius1897} marked the beginning of this line of research by characterizing all linear maps on $n \times n$ complex matrices that preserve the determinant function, that is, $\det \varphi(x) = \det x$ for all $x$. After this important ground-breaking work, linear preserver problems have remained an active and compelling research area in both matrix theory and operator theory, as documented in modern surveys \cite{Molnar2007, GutermanLiSemrl2000,LiPierce2001}. Over the decades, this research topic has expanded significantly, with notable key developments reported in the works \cite{BermanHershkowitzJohnson1985, Marcus1959, JafarianSourour1986, MarcusMoyls1959, Sourour1996, Shitov2021, Shitov2023}.

A fundamental result of Marcus and Moyls~\cite{MarcusMoyls1959} states that every linear map on 
$\mnf$ (with $\mathbb{F}=\mathbb{R}$ or $\mathbb{C}$) that preserves the property of being invertible is necessarily of the form
\begin{equation}\label{standardtransformations}
   \varphi(A) = PAQ \qquad \text{or} \qquad \varphi(A) = PA^{t}Q, 
\end{equation}
for some invertible matrices $P,Q \in \mathrm{GL}_n(\mathbb{F})$. Such mappings are referred to as {\it standard transformations (or transformations of standard form}) \cite{LiPierce2001}.
This finite-dimensional result was extended by Sourour~\cite{Sourour1996} to the infinite-dimensional setting of $L(X)$, the algebra of bounded linear operators on a Banach space $X$, where analogous structure theorems for bijective invertibility-preserving linear maps were obtained. Previously, a partial result of Choi, Jafarian, and Radjavi~\cite{ChoiEtAl1984} was positive linear maps in C$^*$-algebras that preserves invertibility.

The {\it exponential of a matrix $A$} is defined by the infinite series  
\[
e^{A}= \sum_{n=0}^{\infty} \frac{A^n}{n!},
\]
which converges for all square matrices $A$, ensuring that the matrix exponential is well-defined. Given a matrix $B$, a matrix $A$ is called a {\it logarithm of $B$} if it satisfies the equation  
\[
e^{A} = B.
\]

Determining whether a given matrix has a logarithm is a classical problem in matrix analysis. In the complex setting, this problem is well understood and can be summarized as follows: A matrix $A \in \mnc$ has a logarithm if and only if it is invertible \cite[Theorem 2.10]{hall}. Although the logarithm of a matrix may not be unique in general, if the matrix $A$ has no negative real eigenvalues, then it admits a unique logarithm whose eigenvalues $\lambda$ satisfy the bound $|\mathrm{Im}(\lambda)| < \pi$. This unique logarithm is called the {\it principal logarithm} (see also \cite{higham2008functions}).

In the real setting, the existence of a real logarithm for a given real matrix $A$ is a more subtle problem due to the constraints imposed by the real eigenvalues and the structure of the matrix. Culver \cite{culver} provided a complete characterization of when a real matrix possesses a real logarithm. More precisely, the following result gives the necessary and sufficient conditions for the existence of a real logarithm.

\begin{thm}[{\cite[Theorem 1]{culver}}]\label{existreallog}
A real matrix $A \in \mnr$ has a real logarithm, that is, there exists a real matrix $X$ such that $A = e^X$, if and only if the following conditions hold:  
\begin{enumerate}
\item $A$ is invertible.  
\item Each elementary divisor (Jordan block) corresponding to a negative eigenvalue of $A$ occurs an even number of times.
\end{enumerate}
\end{thm}

This class of real matrices admitting a real logarithm plays a fundamental role in various areas of mathematics and its applications. They arise naturally in Lie theory, where the matrix exponential and logarithm connect Lie groups and Lie algebras, providing a bridge between algebraic structures and smooth manifolds. Moreover, real logarithms are essential in control theory, differential equations, and numerical analysis, especially when working with matrix functions that must remain within the real domain.


In this paper, we consider the following two important subsets of $\mnr$:
\[
K_n = \{ A \in \mnr : \text{the principal logarithm } \log(A) \text{ exists} \},
\]
\[
K_n^{*} = \{ A \in \mnr : \exists\, X \in \mnr \text{ such that } e^{X} = A \}.
\]
We may leave off the subscript for the set above when the order is evident from context.

Our main objective is to describe all linear, bijective, maps
\[
\varphi : \mnr \to \mnr
\]
that satisfy \begin{equation}
    \varphi(S) = S,
\end{equation} where $S$ is either $K$ or $K^{*}$. That is, we  characterize all linear operators on $\mnr$
that map each of these two sets onto themselves. In fact, we prove the following general result along these lines.

\begin{thm}\label{thm:K*-preservers}
Let $\varphi : \mnr \to \mnr$ be a bijective linear map.
Then 
\[
\varphi(K^*) = K^*
\]
if and only if there exist $P \in \mathrm{GL}_n(\mathbb{R})$ and $c > 0$ such that either
\[
\varphi(A) = c\, P A P^{-1}
\qquad \text{or} \qquad
\varphi(A) = c\, P A^T P^{-1}, \quad \forall\, A \in \mnr.  \] 
\end{thm}


\section{Main Results}\label{mainresults}

Recall that a matrix $A \in \mnc$ admits a logarithm precisely when it is invertible. 
Hence, every linear operator on $\mnc$ that maps this class of matrices 
(those admitting a logarithm) onto itself is thus an invertibility linear preserver. 
It follows that all such operators are of the standard form described in 
Equation~(\ref{standardtransformations}).

However, the situation is more subtle over the real field.  
In particular, the class of real matrices admitting a real logarithm is not necessarily preserved by all linear maps of the standard form unless additional conditions are satisfied.  
The following example illustrates this difference.
\begin{ex}\label{counterexample}
{\rm Consider the matrix
\[
A = \begin{pmatrix}
-1 & 0 \\
0 & -1
\end{pmatrix} \in \mathbf{M}_2(\mathbb{R}).
\]
Since the eigenvalue $-1$ appears with even multiplicity, $A$ admits a real logarithm.

Define the linear map
\[
\varphi(A) = P A Q, \quad \text{where} \quad 
P = I, \quad
Q = \begin{pmatrix}
0 & 1 \\
1 & 0
\end{pmatrix}.
\]

Then
\[
\varphi(A) = A Q = A Q = \begin{pmatrix}
-1 & 0 \\
0 & -1
\end{pmatrix}
\begin{pmatrix}
0 & 1 \\
1 & 0
\end{pmatrix}
= \begin{pmatrix}
0 & -1 \\
-1 & 0
\end{pmatrix}.
\]
The eigenvalues of $\varphi(A)$ are $1$ and $-1$, each with multiplicity one.
Since the eigenvalue $-1$ appears with odd multiplicity, $\varphi(A)$ does {not} admit a real logarithm.

Thus, $A$ has a real logarithm, but $\varphi(A)$ does not, even though $\varphi$ is of the form $P A Q$ with invertible $P,Q$.}
\end{ex}

\begin{rem}
{\rm In Example~\ref{counterexample}, we have $Q \neq cP^{-1},$ for any $c>0$.  
If we assume $Q = cP^{-1}$ for some $c>0$ in Equation~(\ref{standardtransformations}), then it follows that $A$ admits a real logarithm if and only if $\varphi(A) = cP A P^{-1}$ admits a real logarithm.

This raises a very natural question:  
Is every linear bijection that preserves the set $K^*$ necessarily of the standard form given in Equation~(\ref{standardtransformations}) with $Q = P^{-1}$ (up to positive scaling)?

Before addressing this question, we begin with a lemma that will be helpful in characterizing linear bijections preserving $K^*$ onto itself.}
\end{rem}

We begin by considering a general result concerning linear maps on $\mnr$. For any subset $X \subset \mnr$, we use the standard notation $\overline{X}$ to denote the closure of $X$ with respect to the conventional Euclidean topology on $\mnr$.

\begin{lem}\label{closuremap}
Let $L : \mnr \to \mnr$ be a bijective linear map and let $S \subset \mnr$.  
Then the following implication holds:
\[
L(S) = S 
\quad \implies \quad 
L(\overline{S}) = \overline{S}.
\]
\end{lem}

\begin{proof}
Since $\mnr$ is finite-dimensional, any linear bijection $L$ is a homeomorphism; in particular, both $L$ and $L^{-1}$ are continuous.

\medskip

Because $L$ is continuous,
\[
L(\overline{S}) \subseteq \overline{L(S)} = \overline{S}.
\]
Applying the same reasoning to $L^{-1}$ gives
\[
L^{-1}(\overline{S}) \subseteq \overline{L^{-1}(S)} = \overline{S}.
\]
Applying $L$ to both sides of this inclusion yields
\[
\overline{S} \subseteq L(\overline{S}).
\]
Combining the two inclusions gives
\[
L(\overline{S}) = \overline{S}.
\]
\end{proof}

To achieve this characterization of all linear bijections $\varphi : \mnr \to \mnr$ satisfying
\[
\varphi(K^*) = K^*,
\]
we divide the solution into two key steps:
\begin{enumerate}
    \item First, we analyze the scenario where $\varphi$ is assumed to be of the standard form as described in Equation~(\ref{standardtransformations}). This allows us to understand the behavior and constraints within this well-studied framework.
    \item Next, using Zariski denseness, we establish that any such map $\varphi$ must necessarily preserve $\mathrm{GL}_n(\mathbb{R})$. This crucial step leads us to conclude that $\varphi$ is indeed of standard form. 
\end{enumerate}

It is not difficult to observe that $\overline{K^*} = \overline{K},$
from which it then follows, using Lemma \ref{closuremap}, we will have also determined the linear transformations that preserve the set \( K \). To this end, since
\[
K = \{ A \in \mnr : A \text{ has no negative eigenvalues} \}.
\]
taking closures preserves inclusion, so $\overline{K} \subseteq \overline{K^*}. $

On the other hand, consider any matrix \( A \in K^* \). If all eigenvalues of \( A \) are positive real numbers, then clearly $A \in K \subseteq \overline{K}.$
If \( A \) has a negative eigenvalue, the existence of a real logarithm implies that each negative eigenvalue must have even geometric multiplicity. In this case, there exists a sequence of matrices $\{ A_k \} \subset K$  such that $A_k \to A,$ as the negative eigenvalues of \( A \) can be approximated by complex conjugate pairs of eigenvalues that lie off the negative real axis, ensuring that each \( A_k \in K \).
Thus $A \in \overline{K},$ 
and hence $K^* \subseteq \overline{K}.$ 
Taking closures once again implies
$\overline{K^*} \subseteq \overline{K}.$ Therefore, we conclude that $\overline{K} = \overline{K^*}.$
\begin{rem}
If $A \in \overline{K}$ (or $\overline{K^*}$), then $A$ may have $0$ as an eigenvalue and may also have negative eigenvalues, each occurring with even geometric multiplicity. Consequently, $K$ is not dense in $M_n(\mathbb{R})$, nor is it dense in $\mathrm{GL}_n(\mathbb{R})$ with respect to the subspace topology.
\end{rem}

To the best of our knowledge, this is the first instance in the linear preserver literature where Zariski density is used to obtain such a conclusion. The underlying notions of Zariski topology and Zariski density are classical and have been extensively developed in algebraic geometry and commutative algebra. Their use here provides a clean algebraic bridge: once $K^*$ is shown to be Zariski dense in $\mathrm{GL}_n(\mathbb{R})$ (with respect to the subspace topology), any linear bijection preserving $K^*$ must preserve $\mathrm{GL}_n(\mathbb{R})$, leading directly to the standard form classification.

As is standard, if $A \in \mnf$ and $\alpha \subset \{1,2,\ldots,n\}$ is an index set, then we let $A{[\alpha]}$ denote the principal submatrix of $A$ lying in rows and columns indexed by $\alpha$ in the natural induced order. Furthermore, we let $\operatorname{spec}(A)$ denote the {\it spectrum of $A$} or the multi-set of eigenvalues of $A$.

\subsection{Characterization under the Standard-Form Assumption}\label{Standard-Form}

To begin, we present a characterization of the linear bijections of the form $\varphi(A) = P A Q$ or $\varphi(A) = P A^{T} Q$ that preserve the set $K^*$ onto itself.  
The following theorem shows that such preservers must, up to a positive scalar multiple, be conjugations or transpose-conjugations by an invertible matrix.

\begin{thm}\label{thm:K^*-preservers}
    Let $\varphi: \mnr \to \mnr$ be a bijective linear map of the form 
$\varphi(A) = P A Q$ or $\varphi(A) = P A^{T} Q$. Then
\[
\varphi({K^*}) = {K^*}
\]
if and only if there exist $P \in \mathrm{GL}_n(\mathbb{R})$ and $c > 0$ such that
\[
\varphi(A) = c P A P^{-1} 
\quad \text{or} \quad 
\varphi(A) = c P A^{T} P^{-1}, 
\quad \forall\, A \in \mnr.
\]
\end{thm}

\begin{proof}
($\Rightarrow$) Assume $\varphi(K^*)=K^*$.  
We treat the first case only; the transposed case follows using similar arguments. We divide the proof into sequence of steps for clarity.

\medskip
\noindent
\textbf{Step 1. Reduction to a right-multiplication form.}
Set $M:=QP\in \mathrm{GL}_n(\mathbb R).$ 
Then $P^{-1}\varphi(A)P = AM,$ 
and by $\varphi(K^*)=K^*$ 
we have 
$ A\in {K^*}$ if and only if $AM\in{K^*}, $ 
for some invertible matrix $M.$ Hence applying
Lemma \ref{closuremap}, gives 
\begin{equation}\label{conjugation}
 A\in \overline{K^*}\quad\Longleftrightarrow \quad A M\in \overline{K^*},  
\end{equation}

We now determine the structure of multiplier matrix $M$.

\medskip
\noindent
\textbf{Step 2. Each $2\times2$ principal submatrix of $M$ is invertible.}
Fix $i\ne j$ from $\{1,2,\ldots,n\}$.  
For any invertible $B\in \overline{K^*_2}$, define an $n \times n$ matrix $A$ such that $A{[\{i,j\}]} =B$ and $A$ is zero elsewhere.
Then $A\in\overline{K^*}$, and by Equation (\ref{conjugation}), $AM\in\overline{K^*}$.  
The corresponding $2\times2$ principal block of $AM$ equals $BM{[i,j]}$.  
Since $B$ is invertible, $M{[i,j]}$ must also be invertible (from the rank preserver argument). 
Thus every $2\times2$ principal block of $M$ is invertible. This fact will be crucial in establishing the conclusion of Step 4. In fact, one can prove more generally that every principal submatrix of $M$ is invertible.

\medskip
\noindent
\textbf{Step 3. Positivity of eigenvalues.}
 By setting $A = I$, it follows that $M = I \cdot M \in K^*$, so $M$ admits a real logarithm. In particular, any negative eigenvalue of $M$ must appear with even multiplicity in its corresponding Jordan blocks.
 
\hspace{1cm}\textmd{{\it Step 3.1}: $M$ has no non-real eigenvalues:}
Assume, for the sake of a contradiction, that $M$ has a non-real eigenvalue. Since $M$ is real, we have  $\operatorname{spec}(M) = \{\lambda,\overline{\lambda}\}$, where $\lambda  \notin\mathbb{R}$,
Using the real-Jordan form of a square matrix, it follows that  there is a real $2$-dimensional $M$-invariant subspace
$V\subset\mathbb{R}^n$ and a real basis of $V$, in which the restriction of $M$ is represented by
a real $2\times 2$ matrix of the form
\[
M_0 = r R_\theta
= r\begin{pmatrix}
\cos\theta & -\sin\theta\\
\sin\theta & \cos\theta
\end{pmatrix},~{\rm with} ~ r>0,\ 0<|\theta|<\pi.
\]
Extend this basis of $V$ to a basis of $\mathbb{R}^n$. In the resulting basis, $M$ has the
block upper-triangular form
\[
S^{-1} M S =
\begin{pmatrix}
M_0 & *\\
0 & M_1
\end{pmatrix}
\]
for some $S\in \mathrm{GL}_n(\mathbb{R})$ and some $(n-2)\times(n-2)$ real matrix $M_1$. Observe that $M_1\in K^*$.

We may assume, without loss of
generality, that $M$ is of the above block form, since the spectrum is preserved by similarity transformations.

For the $2\times2$ block $M_0 = rR_\theta$, consider the matrix $A_\theta$:
\[
A_\theta =
\begin{pmatrix}
1 & b_\theta\\
0 & 1
\end{pmatrix},
\]
with $b_\theta\in\mathbb{R}$ chosen such that
$2\cos\theta + b_\theta\sin\theta < -2.$ 
Then $\operatorname{spec}(A_\theta)=\{1,1\}\subset (0,\infty)$, so $A_\theta\in K_2^*$.

Define
\[
B_\theta := A_\theta R_\theta
=
\begin{pmatrix}
1 & b_\theta\\
0 & 1
\end{pmatrix}
\begin{pmatrix}
\cos\theta & -\sin\theta\\
\sin\theta & \cos\theta
\end{pmatrix}
=
\begin{pmatrix}
\cos\theta + b_\theta\sin\theta & -\sin\theta + b_\theta\cos\theta\\
\sin\theta & \cos\theta
\end{pmatrix}.
\]
It easily follows that
\[
\det(B_\theta)
= (\cos\theta + b_\theta\sin\theta)\cos\theta
-(-\sin\theta + b_\theta\cos\theta)\sin\theta
= \cos^2\theta + \sin^2\theta = 1,
\]
\[
\operatorname{tr}(B_\theta)
= (\cos\theta + b_\theta\sin\theta) + \cos\theta
= 2\cos\theta + b_\theta\sin\theta.
\]
Thus, the eigenvalues $\lambda_1,\lambda_2$ of $B_\theta$ satisfy
\[
\lambda_1 + \lambda_2 = \operatorname{tr}(B_\theta), \qquad
\lambda_1\lambda_2 = \det(B_\theta) = 1,
\]
and they are the roots of
\[
\lambda^2 - \operatorname{tr}(B_\theta)\lambda + 1 = 0.
\]
The discriminant is
\[
\Delta = (\operatorname{tr}(B_\theta))^2 - 4\det(B_\theta)
= (2\cos\theta + b_\theta\sin\theta)^2 - 4.
\]
By our choice $2\cos\theta + b_\theta\sin\theta < -2$, we have
$|\operatorname{tr}(B_\theta)|>2$, hence $\Delta>0$. Therefore, $\lambda_1,\lambda_2$ are
real and distinct. Moreover,
\[
\lambda_1\lambda_2 = 1>0, \qquad \lambda_1 + \lambda_2 = \operatorname{tr}(B_\theta) < 0,
\]
so $\lambda_1,\lambda_2$ must have the same sign and a negative sum, which implies $ 
\lambda_1,\lambda_2 \in (-\infty,0). $ 
Thus $B_\theta$ has two distinct negative real eigenvalues.

Now, since $M_0 = rR_\theta$ with $r>0,$ we have 
\[
A_\theta M_0 = A_\theta (rR_\theta) = r B_\theta.
\]
The positive scalar $r$ does not change the sign pattern of eigenvalues, so $A_\theta M_0$
also has two distinct negative real eigenvalues. Therefore $A_\theta M_0\notin K_2^*$.

Embed $A_\theta$ into $\mnr$ by defining
\[
\widehat{A} := A_\theta \oplus I_{n-2}
\]
in the chosen basis. Then all eigenvalues of $\widehat{A}$ are equal to $1$, so
$\widehat{A}\in K^* \subset \overline{K^*}$. However,
\[
\widehat{A} M
=
(A_\theta \oplus I_{n-2})
\begin{pmatrix}
M_0 & *\\
0 & M_1
\end{pmatrix}
=
\begin{pmatrix}
A_\theta M_0 & *\\
0 & M_1
\end{pmatrix}.
\]
The leading $2\times2$ block $A_\theta M_0$ has two distinct negative eigenvalues, so
$\widehat{A}M\notin K^*$ since $M_1 \in K^*$. In particular, $\widehat{A}M\notin \overline{K^*}$.

Thus, we have found $\widehat{A}\in \overline{K^*}$ with $\widehat{A}M\notin\overline{K^*}$, contradicting
Equation (\ref{conjugation}). Therefore, $M$ cannot have non-real eigenvalues, that is $\operatorname{spec}(M)\subset\mathbb{R}.$

\hspace{1cm}\textmd{{\it Step 3.2}: Exclude negative real eigenvalues:} 
Assume $M$ has a negative real eigenvalue $\mu < 0$. Let $V_\mu$ be its generalized eigenspace. Over $\mathbb{R}$, $M|_{V_\mu}$ can be written as a direct sum of Jordan blocks:
\[
M|_{V_\mu} = J_{r_1}(\mu) \oplus J_{r_2}(\mu) \oplus \dots \oplus J_{r_k}(\mu),
\]
(up to similarity).

Since $M \in K^*$, the number of Jordan blocks of each size is even. Construct the  block-diagonal matrix $A$ as follows:
\begin{itemize}
    \item For each generalized eigenspace $V_\lambda$ for eigenvalues $\lambda \neq \mu$, set $A|_{V_\lambda} = I$.
    \item On $V_\mu$, set $A|_{J_{r_m}(\mu)} = c_m I_{r_m}$, with $c_m > 0$ and distinct within each pair.
\end{itemize}

Then $A \in K^*$ because all eigenvalues are positive. Compute $AM$ restricted to $V_\mu$: each Jordan block has eigenvalue $c_m \mu < 0$. Since $c_m$ are distinct within each pair, some negative eigenvalues appear with odd multiplicity (multiplicity 1 per block). Using the real logarithm criterion, $AM \notin K^*$, which is a contradiction.

Thus, all eigenvalues of $M$ are positive real numbers, namely $
\operatorname{spec}(M) \subset (0, \infty).$
 
\medskip
\noindent
\textbf{Step 4. The structure of each $2\times2$ principal submatrix.}
For $i \neq j$ fixed, write
\[
M{[\{i,j\}]}=\begin{pmatrix} a & b\\ c & d\end{pmatrix},~ {\rm with} ~ a,b,c,d\in\mathbb R.
\]

For $\theta\in(0,\pi)$ define the rotation
\[
R(\theta)=
\begin{pmatrix}\cos\theta & -\sin\theta\\[4pt]\sin\theta & \cos\theta\end{pmatrix}.
\]
Each $R(\theta)$ has eigenvalues $e^{\pm i\theta}$, hence $R(\theta)\in\overline{K^*_2}$.  
Using Equation (\ref{conjugation}), $R(\theta)M{[\{i,j\}]}\in\overline{K^*_2}$ for all $\theta$. As before, let $A$ be an $n \times n$ matrix with $A{[\{i,j\}]} =R(\theta)$ and all other entries of $A$ equal zero. Then $AM\in \overline{K^*}.$
The corresponding $2\times2$ principal block of $AM$ equals $R(\theta)M{[\{i,j\}]}.$ For this particular choice of $A$ $$AM\in \overline{K^*}\iff R(\theta)M{[\{i,j\}]}\in\overline{K^*_2}.$$ Compute
\[
C(\theta):=R(\theta)M{[\{i,j\}]}
=\begin{pmatrix}
a\cos\theta - c\sin\theta & b\cos\theta - d\sin\theta\\[4pt]
a\sin\theta + c\cos\theta & b\sin\theta + d\cos\theta
\end{pmatrix}.
\]
Then
\[
\operatorname{tr}C(\theta)=(a+d)\cos\theta+(b-c)\sin\theta,
\qquad
\det C(\theta)=ad-bc>0,
\]
and the discriminant
\[
\Delta(\theta)
=\big(\operatorname{tr}C(\theta)\big)^2-4\det C(\theta).
\]

Define
\[
R_0:=\sqrt{(a+d)^2+(b-c)^2}, \qquad 
\cos\phi=\frac{a+d}{R_0},\quad \sin\phi=\frac{b-c}{R_0}.
\]
Then
\[
\operatorname{tr}C(\theta)=R_0\cos(\theta-\phi),
\quad
\Delta(\theta)=R_0^2\cos^2(\theta-\phi)-4(ad-bc).
\]
Hence
\[
\max_{\theta}\Delta(\theta)
=R_0^2-4(ad-bc)=(a-d)^2+(b+c)^2.
\]

If $(a-d)^2+(b+c)^2>0$, choose $\theta$ close to $\pi$ so that $\operatorname{tr}C(\theta)<0$ and $\Delta(\theta)>0$.  
Then $C(\theta)$ has positive determinant, negative trace, and distinct real eigenvalues, both negative. This contradicts $C(\theta)\in\overline{K^*_2}$.  
Therefore
\[
(a-d)^2+(b+c)^2=0 \quad\Rightarrow\quad a=d,\ b=-c.
\]
Thus
\[
M_{[i,j]}=aI_2+bJ,\quad J=\begin{pmatrix}0&1\\-1&0\end{pmatrix}.
\]

\medskip
\noindent
\textbf{Step 5. Positivity of eigenvalues implies that $b=0$.}
The eigenvalues of $M{[\{i,j\}]}=aI_2+bJ$ are $a\pm ib$.  
Since all eigenvalues of $M$ are real and positive, we must have $b=0$,  
so $M{[\{i,j\}]}=aI_2$ with $a>0$.  
As this holds for every pair $(i,j)$, all diagonal entries coincide and
\[
M=cI_n, \quad c>0.
\]

\medskip
\noindent
\textbf{Step 6. Characterization of $\varphi$.}
If $\varphi(A)=PAQ$, then $Q P=M=cI_n$, so $Q=cP^{-1}$ and
\[
\varphi(A)=c\,PAP^{-1}.
\]
If $\varphi(A)=PA^TQ$, the same argument yields
\[
\varphi(A)=c\,PA^TP^{-1}.
\]

\medskip
\noindent
($\Leftarrow$) The converse is immediate: both forms clearly map $K^*$ onto itself, since similarity, transpose, and positive scaling preserve the spectrum and thus the set $K^*$.
\end{proof}

\begin{thm}\label{thm:K-preserversstandar}
Let $\varphi: \mnr \to \mnr$ be a bijective linear map of the form 
$\varphi(A) = P A Q$ or $\varphi(A) = P A^{T} Q$. Then
\[
\varphi({K}) = {K}
\]
if and only if there exist $P \in \mathrm{GL}_n(\mathbb{R})$ and $c > 0$ such that
\[
\varphi(A) = c P A P^{-1} 
\quad \text{or} \quad 
\varphi(A) = c P A^{T} P^{-1}, 
\quad \forall\, A \in \mnr.
\]
\end{thm}
\begin{proof}
    The proof follows the same steps as in Theorem~\ref{thm:K^*-preservers}
\end{proof}

\subsection{Zariski Denseness and Preservation of \(\mathrm{GL}_n(\mathbb{R})\)}

Theorem~\ref{thm:K-preserversstandar} characterizes all bijective linear maps $\varphi$ satisfying $\varphi(K)=K$ when $\varphi$ is assumed to have a standard form. Our aim now is to classify all bijective linear maps
\[
\varphi : \mnr \to \mnr
\]
satisfying $\varphi(K)=K$ {without} imposing any structural assumptions on $\varphi$.

We claim that every such map necessarily preserves invertibility:
\[
\varphi(\mathrm{GL}_n(\mathbb{R})) = \mathrm{GL}_n(\mathbb{R}),
\]
which would permit a complete classification of $\varphi$ using the existing theory of invertibility-preserving linear transformations.

A natural path towards proving this claim is as follows.  
First we establish that $K$ is {dense} in $\mathrm{GL}_n(\mathbb{R})$ with respect to a suitable topology on $\mnr$, and then it will follow that  $\varphi$ is continuous with respect to this topology. From there the condition that  $\varphi(K)=K$, together with Lemma~\ref{closuremap}, would ensure that the map $\varphi$ must leave the group $\mathrm{GL}_n(\mathbb{R})$ invariant.

To this end, we consider the Zariski topology on the set $\mnr$, which represents a topology particularly well-suited for the study of polynomial equations in algebraic geometry. {For a detailed treatment of the Zariski topology on an affine space, see \cite{reid1988}, Section 3.5.}

With respect to the Zariski topology on $\mnr$, we will establish two key claims:
\begin{enumerate}
    \item[(i)] the set $K$ is dense in $\mathrm{GL}_n(\mathbb{R})$ (with respect to the subspace topology
), and
    \item[(ii)] every linear map $\varphi : \mnr \to \mnr$ is continuous.
\end{enumerate}

\subsection*{Zariski Topology}

A subset $V \subset \mnr$ is called {Zariski closed} if there exist $k$ polynomials
\[
f_1,\dots,f_k \in \mathbb{R}[x_{11},\dots,x_{nn}]
\]
such that
\[
V = \{A \in \mnr : f_1(A)=\cdots=f_k(A)=0\}.
\]
The complement of a Zariski-closed set is defined to be {Zariski open}.
For any subset $S \subseteq \mnr$, its {Zariski closure} is
\[
\overline{S}^{\,Z} := \bigcap_{\substack{V\text{ Zariski closed}\\ S\subseteq V}} V.\] 
A set $S$ is {Zariski dense} in $\mathrm{GL}_n(\mathbb{R})$ if $\overline{S}^{\,Z} = \mathrm{GL}_n(\mathbb{R}).$ 

\subsection*{Euclidean-Open Sets Are Zariski Dense}

A fundamental fact needed here is:

\begin{quote}
A nonzero real polynomial cannot vanish on a nonempty Euclidean-open subset of 
$\mathbb{R}^m$.
\end{quote}

Indeed, if a polynomial vanishes on an open ball, then by real-analytic continuation
it must be the zero polynomial.

Let $V \subsetneq \mathrm{GL}_n(\mathbb{R})$ be a proper Zariski-closed set.  
Then $V = Z(p_1,\dots,p_k)$ for some polynomials $p_i$, at least one of which is nonzero.
Such a polynomial cannot vanish on any nonempty Euclidean-open set.\\

To prove that $K$ is open in $\mnr$, we begin by recalling the following fundamental fact: the eigenvalue map is continuous in the standard Euclidean  topology. That is, small perturbations of a matrix lead to small changes in its eigenvalues.

\subsection*{Continuity of the Eigenvalue Map on {$\mnr$}}

Let $\mnc$ denote the space of all complex $n\times n$ matrices,
equipped with any fixed matrix norm $\|\cdot\|$.
We write $\mathbb{C}^n_{\!\sim}$ for the space of unordered $n$--tuples of
complex numbers, endowed with the quotient topology induced from
$\mathbb{C}^n$ (equivalently, with the usual matching metric; see
\cite[p.~153]{Bhatia1997}).
For $A\in \mnc$, let $\sigma(A)\in\mathbb{C}^n_{\!\sim}$ denote
the multiset of eigenvalues of $A$, counted with algebraic multiplicity.

It is a classical fact (see, e.g., \cite[pp.~138--140]{Artin2011},
\cite[p.~121]{HornJohnson2013}, \cite{Bhatia1997}) that the eigenvalue map
\[
    \sigma : \mnc \longrightarrow \mathbb{C}^n_{\!\sim}
\]
is continuous. More precisely, for every $A\in \mnc$ and every
$\varepsilon>0$, there exists $\delta>0$ such that
\[
    \|A-\widetilde A\| < \delta
    \qquad\Longrightarrow\qquad
    d\bigl(\sigma(A),\sigma(\widetilde A)\bigr) < \varepsilon,
\]
where $d$ denotes the matching distance on $\mathbb{C}^n_{\!\sim}$.

We now extend this continuity result to real matrices.
Identifying $\mnr$ as a real linear subspace of
$\mnc$, consider the inclusion map
\[
    \iota : \mnr \hookrightarrow \mnc, \qquad
    \iota(A) = A.
\]
This map is continuous, and $\mnr$ is a closed subset of
$\mnc$.
Since $\sigma : \mnc \to \mathbb{C}^n_{\!\sim}$ is continuous,
its restriction to $\mnr$ remains continuous.

\begin{pro}
The eigenvalue map
\[
    \sigma : \mnr \longrightarrow \mathbb{C}^n_{\!\sim}
\]
is continuous.
\end{pro}

\begin{proof}
The map $\sigma|_{\mnr}$ is the composition
\[
    \mnr \xrightarrow{\;\iota\;}
    \mnc \xrightarrow{\;\sigma\;}
    \mathbb{C}^n_{\!\sim}.
\]
As both $\iota$ and $\sigma$ are continuous, so is their
composition.  Therefore, the eigenvalue map is continuous on
$\mnr$.
\end{proof}

We are now ready to prove that \(K\) is an open subset of \(\mnr\). Note that
\[
K
=
\{ A \in \mnr : \sigma(A) \cap (-\infty, 0] = \varnothing \}.
\]

\begin{pro}
The set $K$ is open in $\mnr$ with respect to any matrix norm.
\end{pro}

\begin{proof}
Take any matrix $A \in K$. Since $\sigma(A)$ is finite and disjoint from the closed set $(-\infty,0]$, the positive quantity
\[
\delta_0 := \min_{\lambda \in \sigma(A)} \mathrm{dist}(\lambda, (-\infty,0]) > 0
\]
is well-defined. Set $\varepsilon := \frac{\delta_0}{2} > 0$.

By continuity of the eigenvalue map $\sigma: \mnr \to \mathbb{C}^n_{\sim}$, there exists $\eta > 0$ such that
\[
\|A - \widetilde{A}\| < \eta \quad \Longrightarrow \quad d(\sigma(A), \sigma(\widetilde{A})) < \varepsilon.
\]

Hence, for every eigenvalue $\mu$ of $\widetilde{A}$, there exists $\lambda \in \sigma(A)$ such that
\[
|\mu - \lambda| < \varepsilon.
\]
Therefore,
\[
\mathrm{dist}(\mu, (-\infty, 0]) \geq \mathrm{dist}(\lambda, (-\infty, 0]) - |\mu - \lambda| > \delta_0 - \varepsilon = \varepsilon > 0.
\]

It follows that $\sigma(\widetilde{A}) \cap (-\infty, 0] = \varnothing$, i.e., $\widetilde{A} \in K$. This shows that the open ball $\{ \widetilde{A} : \|A - \widetilde{A}\| < \eta \}$ is contained in $K$.
Since $A \in K$ was arbitrary, $K$ is open.
\end{proof}

\subsection*{Zariski Density of $K$}

Assume, for the sake of a contradiction, that there exists a proper Zariski-closed set 
$V \subsetneq \mathrm{GL}_n(\mathbb{R})$ such that $K \subseteq V$.
Then some nonzero polynomial must vanish on $K$.
Since $K$ is Euclidean-open, this polynomial would must vanish on an open set, and
hence everywhere, which is impossible.

Therefore no proper Zariski-closed nonempty subset of $\mathrm{GL}_n(\mathbb{R})$ contains $K$.
Consequently,
\[
\overline{K}^{\,Z} = \mathrm{GL}_n(\mathbb{R}),
\]
which means
$K$ is Zariski dense in  $\mathrm{GL}_n(\mathbb{R}).$





\begin{lem}
Let $\phi : \mnr \to \mnr$ be a linear map. Then
$\phi$ is continuous with respect to the Zariski topology. In particular,
if $\phi$ is bijective, it is a Zariski homeomorphism.
\end{lem}

\begin{proof}
Identify $\mnr \cong \mathbb{R}^{n^2}$ via the coordinates
$(a_{ij})_{1 \le i,j \le n}$. Since $\phi$ is linear, each entry of
$\phi(A)$ is a linear combination of the entries of $A$: for each
$1 \le k,l \le n$ there exist constants $c_{ij}^{kl} \in \mathbb{R}$ such that
\[
\phi(A)_{kl} = \sum_{i=1}^n \sum_{j=1}^n c_{ij}^{kl} \, a_{ij}.
\]
Thus each coordinate function $A \mapsto \phi(A)_{kl}$ is a polynomial
function of degree $1$ in the variables $\{a_{ij}\}$. Hence $\phi$ is a
{polynomial map} $\mnr \to \mnr$.

Now let $V \subset \mnr$ be a Zariski-closed set. Then
$V = Z(f_1,\dots,f_r)$ for some polynomials
$f_1,\dots,f_r \in \mathbb{R}[x_{11},\dots,x_{nn}]$, that is,
\[
V = \{ B \in \mnr : f_1(B) = \cdots = f_r(B) = 0 \}.
\]
Consider the preimage of $V$ under $\phi$:
\[
\phi^{-1}(V)
= \{ A \in \mnr : f_i(\phi(A)) = 0 \text{ for all } i \}.
\]
Since each $\phi(A)_{kl}$ is a polynomial in the entries $a_{ij}$, and
each $f_i$ is a polynomial in the coordinate functions on the target,
the composition $f_i \circ \phi$ is again a polynomial in the variables
$\{a_{ij}\}$. Therefore,
$\phi^{-1}(V) = Z(f_1 \circ \phi,\dots,f_r \circ \phi) $ 
is a common zero set of polynomials, hence Zariski-closed.

Thus the preimage of any Zariski-closed set under $\phi$ is Zariski-closed,
so $\phi$ is continuous in the Zariski topology. If $\phi$ is bijective,
the same argument applies to $\phi^{-1}$, which is also linear (hence
polynomial), so $\phi^{-1}$ is Zariski-continuous and $\phi$ is a
Zariski homeomorphism.
\end{proof}

We are in a position to provide an affirmative answer to our initial claims and thus verify are main results.

\begin{lem}
   Let $\varphi : \mnr \to \mnr$ be a bijective linear map such that
\[
\varphi(K) = K.
\] Then \[
\varphi(\mathrm{GL}_n(\mathbb{R})) = \mathrm{GL}_n(\mathbb{R}).
\]
\end{lem}

A similar result holds for $K^*$, as seen in the next result.

 \begin{lem}
   Let $\varphi : \mnr \to \mnr$ be a bijective linear map such that
\[
\varphi(K^*) = K^*.
\] Then \[
\varphi(\mathrm{GL}_n(\mathbb{R})) = \mathrm{GL}_n(\mathbb{R}),
\]
\end{lem}   
\begin{proof}
Since $K \subseteq K^{*} \subseteq \mathrm{GL}_n(\mathbb{R})$ and $\overline{K}^{\,Z} = \mathrm{GL}_n(\mathbb{R})$, we have
\[
\mathrm{GL}_n(\mathbb{R}) = \overline{K}^{\,Z} \subseteq \overline{K^{*}}^{\,Z} \subseteq \mathrm{GL}_n(\mathbb{R}),
\]
hence $\overline{K^{*}}^{\,Z} = \mathrm{GL}_n(\mathbb{R})$.  
Because $\varphi$ is bijective and preserves $K^*$, $\varphi$  preserves its Zariski closure; thus
\[
\varphi(\mathrm{GL}_n(\mathbb{R})) = \mathrm{GL}_n(\mathbb{R}).
\]
\end{proof}

\section{Conclusion}\label{conclusion}

In this paper, we have provided a complete characterization of all linear bijection maps on \(\mnr\) that preserve the set of matrices admitting a real logarithm. Our main theorem establishes that such preservers must be of the form: either conjugation by an invertible matrix or conjugation composed with transposition scaled by a positive constant. To achieve this characterization, we analyzed the linear maps of the standard form, to understand the behavior on certain relevant sets. Second, leveraging a Zariski denseness result, we show that any linear bijection preserving the set of matrices with real logarithms must preserve \(\mathrm{GL}_n(\mathbb{R})\), which forces the map to be of the classical form.

A key novelty of our approach is the use of Zariski density, which shows that the set of matrices with real logarithms is dense in \(\mathrm{GL}_n(\mathbb{R})\) in the Zariski topology. This algebraic-geometric viewpoint allows confirmation that any linear map preserving this set also preserves invertibility.

More generally, these results open new avenues for the application of algebraic geometric methods, such as Zariski topology, in the study of linear preserver problems. In particular, one can prove that if a subset \(U \subseteq \mathrm{GL}_n(\mathbb{F})\) (for \(\mathbb{F} = \mathbb{R}\) or \(\mathbb{C}\)) is open in the usual topology of \(\mnf\), then any linear map on \(\mnf\) that maps \(U\) into itself must also preserve the full general linear group \(\mathrm{GL}_n(\mathbb{F})\).

\section*{Acknowledgments}
\addcontentsline{toc}{section}{Acknowledgments}
S.M.\ Fallat is supported in part by an NSERC Discovery Research Grant, Application No.: RGPIN-2025-05272.
 The work of the PIMS Postdoctoral Fellow S.\ Mondal leading to this publication was supported in part by the Pacific Institute for the Mathematical Sciences.

\end{document}